\documentclass[
11pt, 
a4paper, 
twoside, 
headinclude,footinclude, 
]{amsart}

\usepackage[T1]{fontenc} 
\usepackage[utf8]{inputenc} 
\usepackage{graphicx} 
\graphicspath{{Figures/}} 
\usepackage{subfig} 
\usepackage{amsmath,amssymb,amsthm,stmaryrd,mathtools} 
\numberwithin{equation}{section}
\usepackage{varioref} 
\usepackage[shortlabels]{enumitem} 
\usepackage{tikz}
\usetikzlibrary{calc}
\usepackage{shuffle}
\usepackage{comment}
\usepackage{hyperref}

\usetikzlibrary{matrix}
\usetikzlibrary{positioning, decorations.markings, decorations.pathmorphing}
\usetikzlibrary{arrows.meta}
\usetikzlibrary{cd}

\theoremstyle{plain} 
\newtheorem{theorem}{Theorem}[section]
\newtheorem*{thma}{Theorem A}
\newtheorem*{thmb}{Theorem B}

\newtheorem{proposition}[theorem]{Proposition}
\newtheorem{corollary}[theorem]{Corollary}
\newtheorem{lemma}[theorem]{Lemma}

\newtheorem*{question*}{Question}

\theoremstyle{definition} 
\newtheorem{definition}[theorem]{Definition}
\newtheorem{example}[theorem]{Example}

\newtheorem{setting}[theorem]{Setting}

\theoremstyle{remark} 
\newtheorem{remark}[theorem]{Remark}

\newcommand{\Z}{\mathbb Z}
\newcommand{\N}{\mathbb N}

\newcommand{\DD}{\mathsf D}
\newcommand{\TT}{\mathsf T}
\newcommand{\Sub}{\mathsf S}
\newcommand{\E}{\mathbb E}

\newcommand{\Se}{\mathbb{S}}
\newcommand{\bo}{\operatorname{b}\nolimits}

\DeclareMathOperator{\CM}{\mathsf{CM}}
\DeclareMathOperator{\grCM}{\mathsf{CM}^{\Z}}

\DeclareMathOperator{\grsCM}{\underline{\CM}^{\Z}}
\DeclareMathOperator{\Gr}{\mathsf{Gr}}
\DeclareMathOperator{\gr}{\mathsf{gr}}
\DeclareMathOperator{\Proj}{\mathsf{Proj}}
\DeclareMathOperator{\grproj}{\mathsf{gr proj}}

\DeclareMathOperator{\qgr}{\mathsf{qgr}}

\newcommand{\Hom}{\operatorname{Hom}\nolimits}
\newcommand{\End}{\operatorname{End}\nolimits}
\newcommand{\Ext}{\operatorname{Ext}\nolimits}

\newcommand{\diag}{\operatorname{diag}\nolimits}

\newcommand{\GrAut}{\operatorname{GrAut}\nolimits}
\newcommand{\chare}{\operatorname{char}\nolimits}

\newcommand{\injdim}{\mathop{{\rm inj.dim}}\nolimits}
\newcommand{\gldim}{\mathop{{\rm gl.dim\,}}\nolimits}
\newcommand{\otimesk}{\otimes_k}

\newcommand{\op}{\operatorname{op}\nolimits}
\newcommand{\bd}{\DD^{\bo}}

\newcommand{\sbd}{\underline{\DD}^{\bo}}
\newcommand{\Mod}{\mathsf{Mod}\,}

\newcommand{\RHom}{\mathbf{R}\strut\kern-.2em\operatorname{Hom}\nolimits}

\newcommand{\uRHom}{\mathbf{R}\strut\kern-.2em\underline{\operatorname{Hom}\nolimits}}

\newcommand{\sing}{\DD_{\mathsf{\tiny Sg}}^{\mathsf{\tiny gr}}}
\newcommand{\xsing}{\DD_{\mathsf{\tiny Sg}}}
\newcommand{\q}{\mathsf{\bf {q}}}

\DeclareMathOperator{\add}{\mathsf{add}}

\DeclareMathOperator{\topm}{top}

\author{Louis--Philippe Thibault}
\address{Institutt for matematiske fag, NTNU, 7491 Trondheim, Norway}
\email{louis.thibault@ntnu.no}

\begin{document}

\title{Tilting objects in singularity categories and levelled mutations}



\begin{abstract}
We show the existence of tilting objects in the singularity category $\sing(eAe)$ associated to certain noetherian AS-regular algebras $A$ and idempotents $e$. This gives a triangle equivalence between $\sing(eAe)$ and the derived category of a finite-dimensional algebra. In particular, we obtain a tilting object if the Beilinson algebra of $A$ is a levelled Koszul algebra. This generalises the existence of a tilting object in $\sing(S^G)$, where $S$ is a Koszul AS-regular algebra and $G$ is a finite group acting on $S$, found by Iyama--Takahashi and Mori--Ueyama.  Our method involves the use of Orlov's embedding of $\sing(eAe)$ into $\bd(\qgr eAe)$, the bounded derived category of graded tails, and of levelled mutations on a tilting object of $\bd(\qgr eAe)$.
\end{abstract}

\thanks{2010 {\em Mathematics Subject Classification.} 16S38, 16G50, 18E30, 16S35}
\thanks{{\em Key words and phrases.} Singularity category, tilting theory,  AS-regular algebra, Gorenstein ring, levelled algebra, Koszul algebra, exceptional sequence, triangulated category, Cohen--Macaulay module}
\maketitle
\tableofcontents

\setcounter{tocdepth}{2}


\section{Introduction}

The singularity category $\xsing(X)$ of an algebraic variety $X$ was introduced by Orlov \cite{Orl04} as an invariant which reflects the properties of the singularities of $X$. It draws inspiration from the homological mirror symmetry conjecture \cite{Kon95}. Analogous categories were defined for (graded) noetherian algebras $R$. In this paper we are interested in the graded singularity category, defined as the Verdier localisation   
\[
	\sing(R):= \bd(\gr R)/\bd(\grproj R).
\]
When $R$ is Gorenstein, this category is of particular importance because it is triangle equivalent to $\grsCM(R)$, the stable category of graded Cohen--Macaulay modules \cite{Buc87, Orl04}, which play a central role in representation theory and commutative algebra (see \cite{Yos90} or \cite{LW12} for nice treatments). \\

Tilting theory is a generalisation of Morita theory, and is a powerful tool in the study of triangulated categories. In fact, tilting objects induce a triangle equivalence between algebraic triangulated categories and the derived category of finite-dimensional algebras \cite{Kel94}. It is thus natural to ask for which graded Gorenstein ring $R$ the singularity category $\sing(R)$ admits tilting objects. This question is very hard to answer in its full generality and has motivated a lot of interesting research over the last years (e.g. \cite{BIY18, DL16a, DL16b, Han19, HIMO14, Kim18,  KST07, KST09, IO13, Yam13}). In this paper, we specialise to the case where $R = eAe$ is an AS-Gorenstein algebra coming from a noetherian AS-regular algebra $A$ and an idempotent $e$. 

\begin{question*}
Let $A$ be a noetherian AS-regular algebra of Gorenstein parameter $\ell$ and $e$ an idempotent. When does $\sing(eAe)$ admit a tilting object? 	
\end{question*}

In this context, tilting objects were discovered in two different settings. The first one was obtained in \cite{AIR15} in the case where $eA_0(1-e) = 0$ and $A$ is bimodule Calabi--Yau of Gorenstein parameter $1$, that is, $A$ is a preprojective algebra over a higher representation-infinite algebra. The second one was given in \cite{IT13} and later generalised in \cite{MU16} in the case where $A=S\#G$ is the skew-group algebra over a Koszul AS-regular algebra $S$, $G \leq \text{GrAut\,} S$ is a finite group whose elements have homological determinant $1$, and $e=\frac{1}{|G|}\sum_{g\in G} g$. In this case, $eAe\cong S^G$. \\

We note that these two results are in some sense at the two ends of a spectrum. On the one hand, we obtain a tilting object in the case where the Gorenstein parameter is $1$. On the other hand, we get a different object when the Gorenstein parameter is equal to the global dimension. We are thus interested in finding tilting objects for the cases where the Gorenstein parameter is more general. For example, we found in \cite{Thi20} a class of skew-group algebras $k[x_1, \ldots, x_d]\#G$ which do not admit a grading endowing them with the structure of a bimodule $d$-Calabi--Yau algebra of Gorenstein parameter $1$, but that naturally have gradings giving them higher Gorenstein parameter which are less than $d$. It would be interesting to determine whether $\sing(S^G)$ admits tilting objects for these gradings.\\

In \cite{Ami13}, the author showed that the result in \cite{AIR15}, that is, when $A$ is a preprojective algebra, can be explained by studying Orlov's embedding \cite{Orl09}  $\Phi: \sing(eAe)\hookrightarrow\bd(\qgr eAe)$, where the target category is the bounded derived category of graded tails. In fact, by a result in \cite{Min12}, $\bd(\qgr eAe)$ has a tilting object which induces a tilting object in $\sing(eAe)$.  A similar technique was also employed in \cite{Ued08} for preprojective algebras of the form $A = k[x_1, \ldots, x_d]\#G$, where $G$ is a cyclic group satisfying a certain property.\\

Inspired by this, we compare two semiorthogonal decompositions: 
\begin{itemize}
\item \cite{Orl09} $\bd(\qgr eAe) \cong \langle \q eAe, \q eAe(1), \ldots, \q eAe(\ell-1), \Phi(\sing(eAe))\rangle$;
\item \cite{MM11} $\bd(\qgr A) \cong \langle \q A, \q A(1), \ldots, \q A(\ell-1)\rangle\cong\bd(\nabla A)$,
\end{itemize}
where $\q: \gr A\to \qgr A$ is the natural quotient functor and $\nabla A$ is the Beilinson algebra. Note that suitable restrictions on $e$ give us an equivalence $\qgr A\cong \qgr eAe$. \\

Using this strategy, we can extend the equivalence  in \cite{AIR15} between $\sing(eAe)$ and the derived category of a finite-dimensional algebra to certain algebras of Gorenstein parameter $2$. 

\begin{thma}
Let $A$ be a locally finite noetherian $d$-AS-regular algebra of Gorenstein parameter $2$. Let $e=e^2\in A$ be such that 
\begin{enumerate}
\item $A/AeA$ is finite-dimensional;
\item $eAe$ is AS-Gorenstein;
\item $eA_0(1-e) = (1-e)A_0e =0$.
\end{enumerate}
Then there is a triangle equivalence 
\[
	\sing(eAe)\cong \bd((1-\tilde e)(\nabla A)(1-\tilde e)),
\]
where $\tilde e$ is the idempotent induced from $e$ in $\nabla A$. 
\end{thma} 

We can also get a second class of tilting objects by mutating an exceptional collection in $\bd(\qgr eAe)$.  Unfortunately, mutations of tilting objects are often not tilting objects, so in order to obtain a satisfying result, we must restrict to the case where $\nabla A$ is a levelled Koszul algebra. The notions of levelled mutations and levelled algebras were introduced in \cite{Hil95} and they offer a natural generalisation of the results obtained in \cite{Bon89}. In particular, levelled mutations of tilting objects in levelled Koszul algebras behave well.  With this hypothesis, we obtain the following result.

\begin{thmb}
Let $A$ be a locally finite noetherian $d$-AS-regular algebra of Gorenstein parameter $\ell$. Assume that $\nabla A$ is a levelled Koszul algebra. Let $e=e^2\in A$ be such that 
\begin{enumerate}
\item $A/AeA$ is finite-dimensional;
\item $eAe$ is AS-Gorenstein;
\item $eA_0e\cong k$.
\end{enumerate}
Then there is a triangle equivalence 
\[
	\sing(eAe)\cong \bd((1-\tilde e)(\nabla A)^!(1-\tilde e)),
\]
where $(\nabla A)^!$ is the Koszul dual of $\nabla A$. 
\end{thmb}

The Beilinson algebras of the algebras considered in \cite{IT13} and \cite{MU16} are levelled Koszul and their idempotent satisfies our conditions. When we restrict to their setting, we obtain the same triangle equivalence. Our theorem covers also certain examples where the Gorenstein parameter is not $1$ nor equal to the global dimension of $A$.


\subsection*{Notation}

Let $k$ be an algebraically closed field. If $V$ is a vector space, denote by $D(V):= \Hom_k(V, k)$ the $k$-dual. If $R$ is a ring, we denote the opposite ring by $R^{\op}$. Let $\Mod R$ be the category of right modules and $\Proj R$ be the category of projective right modules. Let $\mathsf{Gr}\,R$ be the category of graded right $R$-modules. The categories written with lower case letters denote the respective subcategories of finitely generated modules. Let $\mathsf D(-)$ be the derived category and $\bd(-)$ be the bounded derived category. If $\TT$ is a triangulated category, let $\Hom_{\TT}^{\bullet}(A,B):=\oplus_{i\in\Z}\Hom(A,B[i])$. 


\section*{Acknowledgement}
 
The author would like to thank Mads Hustad Sand\o y for valuable comments. 


\section{Preliminaries}


\subsection{Koszul algebras}

Let $A = \bigoplus_{i\geq 0} A_i$ be a positively graded $k$-algebra such that $S := A_0$ is a finite-dimensional semisimple algebra. Throughout this section, all tensor products are over $S$. We define Koszul algebras, following mostly \cite{BGS96}. Combined with a levelled property, defined in the next subsection, these algebras behave well under mutation, a tool which we will need. 

\begin{definition}
The algebra $A$ is \emph{Koszul} if $S$, considered as a right graded $A$-module, admits a graded projective resolution
\[
	\cdots \to P^2 \to P^1 \to P^0\to S\to 0
\]
such that $P^{\ell}$ is generated in degree $\ell$.
\end{definition}
 
Koszul algebras are \emph{quadratic}, that is, there is a graded isomorphism
\[
	A\cong T_S V/( R ),
\]
where $V$ is a $S$-bimodule placed in degree $1$ and $R$ is a $S$-bimodule such that $R\subset V\otimes V$. Here $( R )$ denotes the $2$-sided ideal generated by $R$.\\

To every Koszul algebra one can associate its Koszul dual. 

\begin{definition}
Let $A$ be a Koszul algebra. The \emph{Koszul dual}, denoted by $A^!$, is the graded algebra 
\[
	A^!:=\bigoplus_{i\in\N}\bigoplus_{j\in\Z}\Ext_{\Gr A}^{i}(A_0, A_0(j)). 
\]
\end{definition}  

We have the following properties of Koszul duals. 

\begin{theorem}
Let $A$ be a Koszul algebra and suppose that all $A_i$ are finitely generated left $A_0$-modules. Then
\begin{itemize}
\item $A^!$ is a Koszul algebra;
\item There is an isomorphism of graded algebras $A\cong (A^!)^!$. 
\end{itemize}
\end{theorem}


\subsection{Levelled exceptional collections and mutations}

We explain the notion of levelled exceptional collections and levelled mutations, which were first studied in \cite{Hil95} as a natural generalisation of concepts introduced in \cite{Bon89}. These will turn out to be essential tools in the proof of our main theorem, the idea being that levelled mutations behave well under a Koszulity assumption. \\

Let $\TT$ be a $k$-linear Krull--Schmidt triangulated category, which is of finite type, that is, for any objects $A, B\in \TT$, the vector space
\[
	\bigoplus_{i\in \Z}\Hom_{\TT}(A, B[i])
\]
is finite-dimensional. Let $\Sub\subset \TT$ be a full triangulated subcategory. The \emph{right orthogonal} to $\Sub$, denoted by $\Sub^{\perp}$, is the full triangulated subcategory 
\[
	\Sub^{\perp}:= \{A\in\TT\,\,|\,\,\Hom_{\TT}(B,A) =0 \text{ for any $B\in\Sub$}\}.
\]
Dually, the \emph{left orthogonal} to $\Sub$, denoted by ${^{\perp}\Sub}$, is the full triangulated subcategory
\[
	{^{\perp}\Sub}:= \{A\in\TT\,\,|\,\,\Hom_{\TT}(A,B) =0 \text{ for any $B\in\Sub$}\}.
\]
The subcategory $\Sub$ is said to be \emph{right admissible} (resp. \emph{left admissible}) if there exists a functor $\TT\to \Sub$, which is right (resp. left) adjoint to an embedding $\Sub \to \TT$.  It is \emph{admissible} if it is both left and right admissible. 

\begin{lemma}[{\cite[Lemma 1.4]{Orl09}}]
\label{lem:orl_quot}
Let $\Sub$ be a full triangulated subcategory of $\TT$. If $\Sub$ is right (resp. left) admissible, then $\TT/\Sub\cong \Sub^{\perp}$ (resp. $\TT/\Sub\cong {^{\perp}\Sub}$). 
\end{lemma} 

For a set $\Omega$ of objects in $\TT$, denote by $\langle \Omega\rangle$ the smallest full triangulated subcategory containing the elements of $\Omega$ and closed under isomorphism and direct summands.  

\begin{definition}
A sequence of full triangulated subcategories $(\Sub_0, \ldots, \Sub_m)$ in $\TT$ is called a \emph{semiorthogonal decomposition} if all $\Sub_i$ are admissible in $\TT$ and there is a sequence of left admissible subcategories $\TT_0=\Sub_0\subset \TT_1\subset\cdots\subset \TT_m = \TT$ such that $\Sub_i$ is left orthogonal to $\TT_{i-1}$ in $\TT_i$. In this case we write $\TT = \langle \Sub_0, \ldots, \Sub_m\rangle$. 
\end{definition}

The simplest semiorthogonal decompositions come from exceptional collections. 

\begin{definition}
An object $E\in \TT$ is said to be \emph{exceptional} if $\Hom_{\TT}(E,E[\ell]) = 0$ when $\ell\not=0$ and $\Hom_{\TT}(E,E) = k$. An \emph{exceptional collection} $\E$ is a sequence of exceptional objects $(E_0,\ldots, E_m)$ in $\TT$ such that if $0\leq i<j\leq m$, then $\Hom_{\TT}(E_j, E_i[\ell]) = 0$ for all $\ell\in\Z$. It is \emph{full} if $\TT = \langle\oplus_{\ell=1}^m E_{\ell}\rangle$. Moreover, it is \emph{strong} if, in addition, $\Hom_{\TT}(E_i, E_j[\ell]) = 0$ for all $i$ and $j$ and $\ell\not =0$. 
\end{definition}

If $\TT$ has a full exceptional collection $\E = (E_0,\ldots, E_m)$, then it admits a semiorthogonal decomposition $(\Sub_0,\ldots, \Sub_m)$, where $\Sub_{\ell} = \langle E_{\ell}\rangle\cong \bd(k)$. \\

To every exceptional collection $\E$, we associate a graded finite-dimensional algebra
\[
	\End(\E):= \bigoplus_{\ell\geq 0}\bigoplus_{j-i=\ell}\Hom_{\TT}(E_i, E_j).
\]
This algebra is finite-dimensional and has finite global dimension. \\

Strong full exceptional collections are closely related to tilting objects. We are also interested in the weaker notion of silting objects. 

\begin{definition}
Let $U\in\TT$ be an object. We say that $U$ is \emph{tilting} (resp. \emph{silting}) if $\Hom_{\TT}(U, U[i]) = 0$ for any $i\not =0$ (resp. $i>0$) and $\langle U\rangle = \TT$. 
\end{definition}

Note that an exceptional sequence $\E = (E_0,\ldots, E_m)$ is full and strong if and only if $\oplus_{i=0}^m E_{i}$ is a tilting object in $\TT$. Tilting objects give a nice equivalence for algebraic triangulated categories. 

\begin{theorem}[{\cite[Theorem 4.3]{Kel94}}]
\label{thm:keller}
Let $\TT$ be an algebraic Krull--Schmidt triangulated category with a tilting object $U$. If $\gldim \End_{\TT}(U)<\infty$, then there is a triangle equivalence 
\[
	\Hom_{\TT}^{\bullet}(U, -): \TT\xrightarrow{\sim} \bd(\End_{\TT}(U)). 
\]
\end{theorem} 

We are interested in mutating exceptional collections coming from certain tilting objects in order to get tilting objects in the singularity categories we study.  

\begin{definition}
Let $E\in\TT$ be an exceptional object and $X\in{^{\perp}\langle E\rangle}$. The \emph{left mutation of $X$ through $E$}, denoted by $L_E(X)\in \langle E\rangle^{\perp}$, is defined up to isomorphism by the triangle 
\[
	L_E(X)\to\Hom^{\bullet}_{\TT}(E,X)\otimes E\xrightarrow{ev} X\to L_E(X)[1],
\]
where $ev$ is the evaluation map. Dually, if $X\in \langle E\rangle {^{\perp}}$, we define the \emph{right mutation of $X$ through $E$}, denoted by $R_E(X)\in{^{\perp}\langle E\rangle }$, by the triangle 
\[
	X\xrightarrow{coev}D\Hom^{\bullet}_{\TT}(X,E)\otimes E\to R_E(X)\to X[1],
\]
where $coev$ is the coevaluation map. If $\E = (E_0,\ldots, E_{\ell})$ is an exceptional collection and $X\in {^{\perp}\E}$, then we define 
\[
	L_{\E}(X):= L_{E_0}\cdots L_{E_{\ell}}(X)\in \E^{\perp}. 
\]
Similarly, if $X\in\E^{\perp}$, we define 
\[
	R_{\E}(X):= R_{E_{\ell}}\cdots R_{E_0}(X)\in ^{\perp}\E.
\]
\end{definition}

\begin{definition}
An exceptional collection $\E = (E_0, \ldots, E_m)$ is \emph{$n$-levelled} if there exists a surjective monotonic map $s: \{0,\ldots, m\} \to \{0,\ldots, n\}$ such that 
\[
	\Hom_{\TT}^{\bullet}(E_i, E_j) = 0 \text{ for all $i\not = j$ for which $s(i) = s(j)$}.
\]  
The subcollections $\E_i:= (E_{i_0}, \ldots, E_{i_{\mu}})_{i_0,\ldots, i_{\mu}\in s^{-1}(i)}$ are called \emph{levels}. We define the \emph{right levelled mutations} on levelled exceptional collections as follows:
\[
	 R_i(\E_1, \ldots, \E_{m}) := (\E_1,\ldots, \E_{i-1}, \E_{i+1}, R_{\E_{i+1}}(\E_i), \E_{i+2},\ldots, \E_m),
\]
where, if $\E_i = (E_{i_0}, \ldots, E_{i_{\mu}})$, then 
\[
	R_{\E_{i+1}}(\E_i) := (R_{\E_{i+1}}(E_{i_0}), \ldots, R_{\E_{i+1}}(E_{i_{\mu}})).
\]
The \emph{left levelled mutations} $L_i$ are defined similarly:
\[
	L_i(\E_1, \ldots, \E_{m}):= (\E_1,\ldots, \E_{i-2}, L_{\E_{i-1}}(\E_i), \E_{i-1}, \E_{i+1},\ldots, \E_m).
\]  
\end{definition}

\begin{proposition}[{\cite[Proposition 4.2]{Hil95}}]
Let $\E$ be a levelled exceptional collection. Then $R_i\E$ and $L_i\E$ are also levelled exceptional. Moreover, if $\E$ is full, then these are also full. 
\end{proposition}

Two exceptional collections associated to a levelled exceptional collection $\E$ are of particular importance. If $\E_i$ is a level and $r\in\N$, define the \emph{iterated right mutation} 
\[
	R^{r}(\E_i):= R_{\E_{i+r}}\cdots R_{\E_{i+1}}(\E_i)\in R_{i+r-1}\cdots R_{i}(\E).
\]
Similarly one can define the \emph{iterated left mutation}:
\[
	L^r(\E_i): L_{\E_{i-r}}\cdots L_{\E_{i-1}}(\E_i)\in L_{i-r+1}\cdots L_{i}(\E).
\]

\begin{definition}
Let $\E = (\E_0, \ldots, \E_n)$ be a levelled exceptional collection. The \emph{levelled right dual} collection is defined as
\[
	\E^{\vee}:=(\E_n, R^1\E_{n-1}, \ldots, R^n\E_0).
\]
The \emph{levelled left dual} collection is defined as 
\[
	^{\vee}\E:=(L^n\E_n, L^{n-1}\E_{n-1}, \ldots, \E_0).
\]
\end{definition}

Unfortunately, mutations do not preserve in general the strong property of exceptional collections. In other words, as opposed to silting objects, mutations of tilting objects are often not tilting objects.   However, if the endomorphism algebra is levelled and Koszul, then the dual collections remain strong.

\begin{definition}
A quiver $Q$ is \emph{ordered} if $Q_0 = \{0,\ldots, m\}$ is an ordered set and for all $i\leq j$, $e_iQ_1e_j = \varnothing$. Moreover, we say that $Q$ is \emph{$n$-levelled} if it is ordered and there exists a surjective monotonic map 
\[
	s: Q_0\to \{1,\ldots, n\}
\]
having the property that if $e_jQ_1e_i\not = \varnothing$, then $s(j) = s(i) +1$. Finally, an algebra is \emph{levelled} if it is Morita equivalent to a quiver algebra with a levelled quiver. 
\end{definition}

If $A$ is an ordered algebra and $\{e_0,\ldots, e_m\}$ is a complete set of primitive orthogonal idempotents in the given ordering, then $\bd(A)$ admits a full strong exceptional collection
\[
	\E = (P_0, \ldots, P_m ),
\]
where $P_i = e_iA$. It is levelled if $A$ is a levelled algebra. Note that in this case $A \cong \End(\E)$.\\

We have the following necessary and sufficient conditions for a levelled algebra to be Koszul. Let $S_i = \topm e_iA$. 

\begin{lemma}[{\cite[Lemma 3.1]{Hil95}}]
\label{lem:kos}
A levelled algebra $A$ is Koszul if and only if $\Ext^{\ell}_A(S_j, S_i))\not = 0$ only for $\ell = s(j)-s(i)$. 
\end{lemma} 

\begin{proposition}[{\cite[Proposition 4.5]{Hil95}}]
\label{prop:kos_dual}
Let $A$ be a levelled algebra with level function $s:\{0,\ldots, m\}\to\{0, \ldots, n\}$. Let $\mathbb P = (P_0, \ldots, P_m)$ be the full strong exceptional collection in $\bd(A)$ consisting of indecomposable projective $A$-modules. Then, 
\[
	^{\vee}\mathbb P = (S_m[-s(m)], S_{m-1}[-s(m-1)], \ldots, S_0),
\]
where $S_i:= \topm P_i$. Moreover, $A$ is Koszul if and only if $^{\vee}\mathbb P$ is a full strong exceptional collection. In this case, $\End(^{\vee}\mathbb P)\cong A^!$.
\end{proposition}

The second part of the proposition is a direct consequence of Lemma \ref{lem:kos}. Now suppose that $\TT$ is an algebraic Krull--Schmidt triangulated category. Let $\E$ be a full strong exceptional collection in $\TT$. Using the equivalence of Theorem \ref{thm:keller},
\[
	\TT\xrightarrow{\sim} \bd(\End(\E)),
\]
we obtain the following corollary. 
\begin{corollary}
\label{cor:kos_dual}
If $\E$ is a full strong exceptional collection in an algebraic Krull--Schmidt triangulated category and $A:=\End(\E)$ is an $n$-levelled algebra, then the left dual collection $^{\vee}\E$ is also a strong full exceptional collection and 
\[
	\End(^{\vee}\E)\cong A^!.
\]  
\end{corollary}

We can also obtain a similar statement for the right dual collection as follows. By \cite[Corollary 2.10]{BS10}, if $E^i\in\E_i$, then
\[
	R^{n-s(i)}(E^i)=\Se_{m}^{-1} (L^{s(i)} (E^i)), 
\]
where $\Se$ is the Serre functor on $\TT$, which exists because $\TT$ has a full exceptional collection \cite[Corollary 2.10]{BK89}. Here, the length of $\E$ is $m+1$ and $\Se_m:= \Se[-m]$. This implies that 
\[
	\Hom^{\bullet}_{\TT}(R^{n-s(i)}(E^i), R^{n-s(j)}(E^j)) \cong \Hom^{\bullet}_{\TT}(L^{s(i)}(E^i), L^{s(j)}(E^j)).
\]
Therefore, if $A$ is levelled Koszul, then $\E^{\vee}$ is also strong and $\End(\E^{\vee}) \cong A^!$ as well. 


\subsection{AS-regular algebras and singularity categories}
In this subsection, we give the definition of the objects that we study in this paper. 

\begin{definition}
Let $A=\oplus_{i\geq 0} A_i$ be a noetherian locally finite graded algebra. We say that $A$ is \emph{$d$-AS-regular} (resp. \emph{$d$-AS-Gorenstein}) of Gorenstein parameter $\ell$ if $\gldim A = d$ and $\gldim A_0 <\infty$ (resp. $\injdim_A  A =\injdim_{A^{\op}} A = d$) and 
\[
	\uRHom_{A}(A_0, A)\cong (DA_0)(\ell)[-d]\quad\text{in $\DD(\mathsf{Gr}\, A_0)$ and in $\DD(\mathsf{Gr}\, A_0^{\op})$},
\]
where $\uRHom_{A}(A_0, A):=\oplus_{i\in \Z}\RHom_{\mathsf{Gr}\, A}(A_0, A(i))$. 
\end{definition}
We now proceed to define the categories that we are interested in. 
\begin{definition}
Let $A$ be a positively graded noetherian algebra. 
\begin{enumerate}
\item We define the quotient abelian \emph{category of graded tails}
\[
	\qgr A := \gr A/\mathsf{tors}\, A,
\]
where $\mathsf{tors}\, A$ is the full subcategory consisting of all graded finite-dimensional $A$-modules. Let $\q: \gr A\to\qgr A$ be the natural quotient functor. The morphisms in $\qgr A$ are given by 
\[
	\Hom_{\qgr A}(\q M, \q N):=\lim_{p\to\infty} \Hom_{\gr A}(M_{\geq p}, N).
\]
\item We define the \emph{singularity category} to be the Verdier localisation 
\[
	\sing(A):= \bd(\gr A)/\bd(\grproj A),
\]
where $\bd(\grproj A)$ is the triangulated subcategory consisting of objects that are isomorphic to bounded complexes of projectives. We denote by $\pi: \bd(\gr A)\to \sing(A)$ the localisation functor.

\item Let $A$ be AS-Gorenstein. We define the \emph{graded category of Cohen--Macaulay $A$-modules} as follows:
\[
	\grCM(A):=\{M\in \gr A\,\, | \,\, \Ext_A^i(M, A)=0\text{ for any $i>0$}\}.
\]
The \emph{stable category} $\grsCM(A)$ has the same objects as $\grCM(A)$ and the morphisms are given by 
\[
	\Hom_{\grsCM(A)}(M, N):= \Hom_{\grCM(A)}(M,N)/[A](M,N),
\]
where $[A](M,N)$ consists of the morphisms which factor through $\add A$.
\end{enumerate}
\end{definition}

A famous theorem of Buchweitz \cite[Theorem 4.4.1]{Buc87} and Orlov \cite[Theorem 3.9]{Orl04} states that if $A$ is AS-Gorenstein, then there is a triangle equivalence
\[
	\sing(A)\cong \grsCM(A). 
\] 
Our leading motivation is to determine when these categories admit tilting objects. By  \cite[Propositions 1.3 \& 1.4]{IT13}, they are Krull--Schmidt algebraic triangulated categories, so the existence of tilting objects implies that they are equivalent to the derived category of a finite-dimensional algebra. The first two categories in our definition are related by the following result of Orlov, which was generalised to our setting in \cite{BS15}. 

\begin{theorem}[{\cite[Theorem 2.5]{Orl09}}]
Let $A$ be an AS-Gorenstein algebra of parameter $\ell\geq 1$. Then there exists a fully faithfull functor 
\[
	\Phi: \sing(A)\to \bd(\qgr A)
\]
and a semiorthogonal decomposition 
\[
	\bd(\qgr A) = \langle \q A, \ldots, \q A(\ell-1), \Phi(\sing(A))\rangle. 
\]
\end{theorem}

When $A$ is AS-regular, there is also a semiorthogonal decomposition of $\bd(\qgr A)$, given in \cite{MM11}. 

\begin{definition}
\label{def:beilin}
Let $A$ be an AS-regular algebra of Gorenstein parameter $\ell$. The \emph{Beilinson algebra} is defined by 
\[
	\nabla A:=\begin{pmatrix}
	A_0& 0& 0 & \cdots & 0 & 0\\
	A_1 & A_0 & 0& \cdots & 0 & 0\\
	A_2 & A_1 & A_0& \cdots & 0& 0\\
	\vdots  & \vdots & \vdots& \ddots & \vdots & \vdots\\
	A_{\ell-2}  & A_{\ell-3} & A_{\ell-4}& \cdots & A_0 & 0\\
	A_{\ell-1} & A_{\ell-2} & A_{\ell-3}& \cdots & A_1& A_0\\
	\end{pmatrix}.
\]
\end{definition}

If $e$ is an idempotent in $A$, then we define $\tilde e_i:= \diag(0,\ldots, 0, e, 0, \ldots, 0)$ to be the diagonal matrix in $\nabla A$ with only one non-zero entry in position $(i,i)$ and we set $\tilde e := \sum_{i=0}^{\ell-1} e_i$. 

\begin{theorem}[{\cite[Proposition 4.3, Proposition 4.4, Theoreom 4.12]{MM11}}]
\label{thm:MM_tilting}
Let $A$ be an AS-regular algebra of Gorenstein parameter $\ell$. There is a semiorthogonal decomposition 
\[
	\bd(\qgr A) = \langle \q A,\ldots, \q A(\ell-1)\rangle.
\]
Moreover, $\q U:= \oplus_{i=0}^{\ell-1} \q A(i)$ is a tilting object in $\bd(\qgr A)$ and $\End_{\bd(\qgr A)}(\q U)\cong\nabla A$. This implies that there is a triangle equivalence 
\[
	\RHom_{\bd(\qgr A)}(\q U, -): \bd(\qgr A)\to \bd(\nabla A). 
\]
\end{theorem}

In the remainder of the paper, we shall assume the following setting. 

\begin{setting}
\label{setting}
Let $A = \oplus_{i\geq 0} A_i$ be a locally finite noetherian $d$-AS-regular algebra of Gorenstein parameter $\ell$, with $\ell\geq 1$. Let $e = e^2\in A$ be such that
\begin{enumerate}[a)]
\item \label{enu:idem_fd} $A/AeA$ is finite-dimensional; 
\item \label{enu:idem_gor} $eAe$ is $d$-AS-Gorenstein of parameter $\ell$. 
\end{enumerate}
\end{setting}

Condition \ref{enu:idem_fd} implies that the functor 
\begin{equation}
\begin{split}
\label{ali:equi_qgr}
	\Psi:\gr A&\to \gr eAe\\
	M &\mapsto Me
\end{split}
\end{equation}
induces an equivalence $\Psi: \qgr A\cong \qgr eAe$ \cite[Proof of Corollary 3.3]{Ami13}. Condition \ref{enu:idem_gor} allows us to invoke Orlov's semiorthogonal decomposition to study $\sing(eAe)$. Note that it follows automatically from \ref{enu:idem_fd} when $A$ is bimodule Calabi--Yau of Gorenstein parameter $\ell$ \cite[Proof of Theorem 4.3]{Ami13}. For more information on bimodule Calabi--Yau algebras, we refer to \cite{AIR15}. An important fact is that they are a special class of AS-regular algebras \cite[Theorem 3.5]{MU16b}. \\

Before stating our main running example, we define skew-group algebras. If $R$ is an algebra and $G<\text{Aut}\, R$ is a subgroup, then the \emph{skew-group} algebra, denoted by $R\#G$, is defined as a vector space by $R\#G = R\otimesk kG$ with multiplication 
\[
	(a\otimes g)(b\otimes h) = ag(b)\otimes gh.
\]
We define $G\#R$ in a similar way. 

\begin{example}
Let $S = k[x_1, \ldots, x_d]$, $G<SL(n,k)$ be finite and $e = \frac{1}{G}\sum_{g\in G} g$. Let $A=S\#G$ be the skew-group algebra. Then there exists many gradings endowing it with a structure of AS-regular algebra, for example by putting the variables in degree $1$. The reader can find other examples of gradings in the next sections. In this case, $eAe\cong S^G$, the invariant ring, and it is well-known that, since $G<SL(d,k)$, the ring $S^G$ is AS-Gorenstein. If, in addition, $S^G$ is an isolated singularity, then $A$ and $e$ satisfy condition \ref{enu:idem_fd} in Setting \ref{setting}.  
\end{example}

In this paper, our examples will be skew-group algebras as described above. We therefore explain here how to construct the relevant quivers. 

\begin{definition}
\label{def:mckay}
Let $S = k[x_1,\ldots, x_d]$ be the polynomial ring and ${G= \frac{1}{r}(a_1, \ldots, a_d)<SL(d,k)}$ be the cyclic group generated by the diagonal matrix $\diag(\xi^{a_1}, \ldots, \xi^{a_d})$, where $\xi$ is an $r$th root of unity and $0\leq a_j<r$. Let $A:=S\#G$ be the skew-group algebra. Then $A\cong kQ/(R)$, where $Q$ is the \emph{McKay quiver}, whose vertices are given by $\Z/r\Z$. Furthermore, there are arrows 
\[
	x_j: i\to i+a_j
\]
for each $i\in \Z/r\Z$ and $1\leq j\leq d$. The relations are generated by $x_ix_j-x_jx_i = 0$. \\

If $A$ is endowed with a grading giving it the structure of an AS-regular algebra of Gorenstein parameter $\ell$, then we can deduce from \cite[Proposition 7.13]{IT13} how to compute the quiver of the finite-dimensional algebra $\nabla A$. Indeed, it is described by the \emph{$\ell$-folded McKay quiver}, whose vertices are given by $\Z/r\Z\times \{0,\ldots,\ell-1\}$. If $x_j$ is an arrow in degree $\delta$ in the quiver of $A$, then we have arrows 
\[
	x_j: (i,p)\to (i + a_j, p+\delta)
\]
for each $0\leq p\leq \ell-1$ such that $p + \delta\leq \ell-1$, $i\in \Z/r\Z$ and $1\leq j\leq d$. We denote $(i, p)$ by $i^p$. The idempotent $\tilde e\in \nabla A$ induced from $e$ is then given by $e^0+\ldots + e^{\ell-1}$. 
\end{definition}

As mentioned in the introduction, two partial answers were given to our motivating question. 

\begin{theorem}[{\cite[Theorem 4.1]{AIR15}}]
\label{thm:AIR}
Let $A$ be a noetherian bimodule Calabi--Yau algebra of Gorenstein parameter $1$. Let $e$ be an idempotent such that 
\begin{enumerate}
\item $A/AeA$ is finite-dimensional;
\item \label{cond:AIR} $eA_0(1-e) = 0$. 
\end{enumerate}
Then $\pi Ae$ is a tilting object in $\sing(eAe)$ and there is a triangle equivalence
\[
	\sing(eAe)\cong\bd((1- e)A_0(1- e)).  
\]	
\end{theorem}

Note that since the Gorenstein parameter is $1$, we have that $A_0= \nabla A$ and $e = \tilde e$. 

\begin{theorem}[{\cite[Theorem 1.7]{IT13};\cite[Theorem 4.17]{MU16}}]
\label{thm:MU}
Let $S$ be a noetherian AS-regular Koszul algebra over $k$ of dimension $d\geq 2$. Let $G \leq \GrAut S$ be a finite subgroup such that $\chare  k$ does not divide $|G|$ and let $e = \frac{1}{|G|} \sum_{g\in G} g$. Assume in addition that 
\begin{enumerate}[a)]
\item $S\#G/( e)$ is finite-dimensional over $k$;
\item $S^G \cong eS\#Ge$ is AS-Gorenstein.
\end{enumerate}
Then there is a triangle equivalence 
\[
	\sing(S^G)\cong \bd((1-\tilde e)(G\# \nabla (S^!))(1-\tilde e)).
\]
\end{theorem}

The fact that $S\#G/( e)$ is finite-dimensional is equivalent to the notion of being a noncommutative graded isolated singularity. Moreover, $S^G$ is AS-Gorenstein if every element of $G$ has homological determinant $1$ (see \cite{JZ00}). Note that, in this setting, the Gorenstein parameter of $S$ is $d$.\\

In both cases the algebras satisfy Setting \ref{setting}. Theorem \ref{thm:AIR} describes a situation in which the Gorenstein parameter is $1$, whereas in Theorem \ref{thm:MU}, the Gorenstein parameter always equals the global dimension. The goal of this paper is to extend these two results by comparing the two semiorthogonal decompositions of $\bd(\qgr eAe)$. In particular, our generalisation will cover certain examples where the Gorenstein parameter is not $1$, nor equal to the global dimension. 

\begin{remark}
One other possible approach to finding tilting objects in $\sing(eAe)$ is to consider the preprojective algebra $\Pi$ over $\nabla A$. From \cite{MM11}, this algebra is an AS-regular algebra of parameter $1$. Moreover, there is an equivalence of categories 
\[
	\Gr A\cong \Gr \Pi. 
\]
It could thus be tempting to apply Theorem \ref{thm:AIR} to find a tilting object. However, the corresponding idempotent $\tilde e\in \Pi$ does not satisfy in general condition (\ref{cond:AIR}) of the statement. Therefore this approach does not work directly. 
\end{remark}


\section{A silting object}

Let $A=\oplus_{i\geq 0} A_i$ and $e=e^2\in A$ be as in Setting \ref{setting}. Let $e' := (1-e)$. In this section, we give a silting object in $\sing(eAe)$. This object is also tilting in some specific cases, which we describe. Let 	
\[
	U:= \bigoplus_{i=0}^{\ell-1} A(i)\in\gr  A
\]  
and define
\[
	\sbd(\qgr eAe):= \bd(\qgr eAe)/\langle \q eAe,\ldots, \q eAe(\ell-1)\rangle.
\]
Using Orlov's semiorthogonal decomposition 
\[
	\bd(\qgr eAe) \cong \langle \q eAe, \ldots, \q eAe(\ell-1), \Phi(\sing(eAe))\rangle,	
\]
we have that
\[
	\Phi(\sing(eAe)) \cong {^{\perp}}\langle \q eAe, \ldots, \q eAe(\ell-1)\rangle.
\]
Then, by Lemma \ref{lem:orl_quot}, we obtain a triangle equivalence 
 \[
 	\sing(eAe)\cong \sbd(\qgr eAe).
 \] 

\begin{theorem}
\label{thm:silting} The object $\q e'Ue$ is a silting object in $\sbd(\qgr eAe)$. It is tilting if either 
\begin{enumerate}[a)]
\item {\cite[Theorem 4.1]{AIR15}, \cite[Theorem 4.3]{Ami13}} $\ell = 1$ and $eA_0e' = 0$ or $e'A_0e =0$, or 
\item $\ell =2$ and $eA_0e' = e'A_0 e = 0$. 
\end{enumerate}
\end{theorem}

\begin{remark}
In \cite[Theorem 1.6]{IT13}, the authors proved b) in the case $A = k[x,y]\#G$, where $k[x,y]$ is the polynomial ring, $G<SL(2,k)$ is finite, $e = \frac{1}{|G|}\sum_{g\in G} g$ and the grading is induced by putting the variables in degree $1$. In their setting, we have that $A_0 = kG$ is semisimple, so $eA_0e' = e'A_0e = 0$. 
\end{remark}

\begin{proof}
By Theorem \ref{thm:MM_tilting}, $\q U$ is a tilting object in $\bd(\qgr A)$. Moreover, recall that the functor 
\begin{equation*}
\begin{split}
	\Psi:\gr A&\to \gr eAe\\
	M &\mapsto Me
\end{split}
\end{equation*}
described in (\ref{ali:equi_qgr}), induces an equivalence of categories 
\[
	\qgr A\xrightarrow{\sim} \qgr eAe. 
\]
Thus, $\q Ue$ is a tilting object in $\bd(\qgr eAe)$. Now, each $\q eAe(i)$ is a projective summand of $\q Ue$. Therefore, by \cite[Theorem 3.6]{IY18}, $\q e'Ue$ is a silting object in 
 \[
 	\sbd(\qgr eAe)\cong \sing(eAe).
 \] 
 
Now assume that $\ell =1$ and $eA_0 e' = 0$, the case where $e'A_0e = 0$ being similar. Then, 
\begin{equation}
\label{eq:hom_vanish}
	\Hom_{\bd(\qgr eAe)} (\q e'Ae, \q eAe) \cong eA_0e' = 0, 
\end{equation}
so the tilting object $\q Ue = \q Ae\in \bd(\qgr eAe)$ gives rise to a semiorthogonal decomposition 
\[
	\bd(\qgr eAe)\cong \langle \q eAe, \q e'Ae\rangle.
\]  
Therefore, $\q e'Ue= \q e'Ae$ is a tilting object in 
\[
	\langle \q e'Ae\rangle\cong {^{\perp}}\langle \q eAe\rangle\cong \bd(\qgr eAe)/\langle \q eAe\rangle = \sbd(\qgr eAe), 
\]
since it is a direct summand of $\q Ue$.\\

Similarly, if $\ell = 2$ and $eA_0e' = e'A_0 e =0$, then, in addition to (\ref{eq:hom_vanish}), we have  
\[
	\Hom_{\bd(\qgr eAe)} (\q eAe, \q e'Ae)\cong e'A_0e = 0.
\]
 Thus, the tilting object $\q Ue = \q Ae \oplus \q Ae(1)$ gives rise to a semiorthogonal decomposition 
 \[
 	\bd(\qgr eAe)\cong \langle \q eAe,\q e'Ae, \q e'Ae(1), \q eAe(1)\rangle,
\]
so $\q e'Ae \oplus \q e'Ae(1) \oplus \q eAe(1)$ is a tilting object in 
\[
	\langle \q e'Ae , \q e'Ae(1) , \q eAe(1)\rangle \cong {^{\perp}}\langle eAe\rangle \cong \bd(\qgr eAe)/\langle\q eAe\rangle  
\]
and $\q e'Ue = \q e'Ae \oplus \q e'Ae(1)$ is a tilting object in 
\[
	\langle \q e'Ae , \q e'Ae(1) \rangle \cong \langle eAe(1)\rangle{^{\perp}}\cong \bd(\qgr eAe)/\langle \q eAe, \q eAe(1)\rangle= \sbd(\qgr eAe),
\]
where the right orthogonal is taken in $\bd(\qgr eAe)/\langle \q eAe\rangle$.
\end{proof}

We now compute the endomorphism ring of the silting object we found.

\begin{lemma}
There is an isomorphism of graded $k$-algebras
\[
	\End_{\sbd(\qgr eAe)}(\q e' U e)\cong (1-\tilde e)(\nabla A)( 1-\tilde e). 
\]
\end{lemma}

\begin{proof}
We have that 
\begin{align*}
	\End_{\sbd(\qgr eAe)}(\q e' U e)&\cong (1-\tilde e)\End_{\sbd(\qgr eAe)}(\q U e)(1-\tilde e)\\
							& \cong(1-\tilde e)(\nabla A)( 1-\tilde e),
\end{align*}
where the last isomorphism is given in Theorem \ref{thm:MM_tilting}.
\end{proof}

\begin{corollary}
If $A$ satisfies hypothesis a) or b) in Theorem \ref{thm:silting}, then there is a triangle equivalence 
\[
	\sing(eAe)\cong\bd((1-\tilde e)(\nabla A)( 1-\tilde e)).
\]
\end{corollary}

\begin{proof}
This is a direct application of Theorem \ref{thm:keller}. The category $\sing(eAe)$ is an algebraic Krull--Schmidt triangulated category. Moreover, $(1-\tilde e)(\nabla A)( 1-\tilde e)$ is an ordered finite-dimensional algebra, so it has finite global dimension. 
\end{proof}

\begin{example}
\label{ex:silting}
Let $A = k[x_1,x_2,x_3]\#G$, where $G = \frac{1}{5}(1,2,2)<SL(3,k)$. We refer to Definition \ref{def:mckay} for explanations on how to construct the following quivers. We have that $A\cong kQ/(R)$, where $Q$ is the McKay quiver

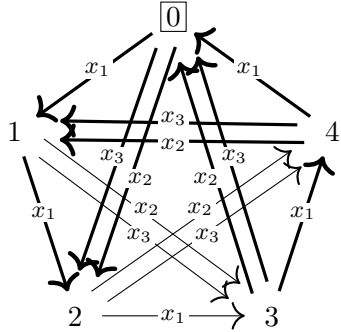
\begin{center}
\begin{tikzpicture}[scale = 1]

    \tikzstyle{ann} = [fill=white,font=\footnotesize,inner sep=1pt]

\def \n {5}
\def \radius {2cm}
\def \second {0.30 cm}
\def \margin {8} 
\def \correction{90}
\def \adjust{56}

\foreach \s in {1,...,3}
{
  \draw[decoration={markings,mark=at position 1 with {\arrow[scale=1.4]{>}}},
    postaction={decorate}, shorten >=0.4pt, very thick] ({360/\n * (\s - 1) + \correction/\n + \margin}:\radius) 
    to ({360/\n * (\s) + \correction/\n-\margin}:\radius); 
    }
      \draw[decoration={markings,mark=at position 1 with {\arrow[scale=2]{>}}},
    postaction={decorate}, shorten >=0.4pt] ({360/5 * (4 - 1) + \correction/5 + \margin}:\radius) 
    to ({360/5 * (4) + \correction/5-\margin}:\radius); 
      \draw[decoration={markings,mark=at position 1 with {\arrow[scale=1.4]{>}}},
    postaction={decorate}, shorten >=0.4pt, very thick] ({360/5 * (5 - 1) + \correction/5 + \margin}:\radius) 
    to ({360/5 * (5) + \correction/5-\margin}:\radius); 

\foreach \s in {2,...,4}
{
  \draw[decoration={markings,mark=at position 1 with {\arrow[scale=1.4]{>}}},
    postaction={decorate}, shorten >=0.4pt, very thick] ({360/\n * (\s +2)+\adjust/\n+\margin-2}:\radius-0.2cm) to ({360/\n * (\s - 1) + \adjust/\n+\margin-4}:\radius-0.4cm);
}
  \draw[decoration={markings,mark=at position 1 with {\arrow[scale=2]{>}}},
    postaction={decorate}, shorten >=0.4pt] ({360/\n * (1 +2)+\adjust/\n+\margin-2}:\radius-0.2cm) to ({360/\n * (1 - 1) + \adjust/\n+\margin-4}:\radius-0.4cm);
      \draw[decoration={markings,mark=at position 1 with {\arrow[scale=2]{>}}},
    postaction={decorate}, shorten >=0.4pt] ({360/\n * (5 +2)+\adjust/\n+\margin-2}:\radius-0.2cm) to ({360/\n * (5 - 1) + \adjust/\n+\margin-4}:\radius-0.4cm);

\foreach \s in {2,...,4}
{
  \draw[decoration={markings,mark=at position 1 with {\arrow[scale=1.4]{>}}},
    postaction={decorate}, shorten >=0.4pt, very thick] ({360/\n * (\s +2)+\adjust/\n+\margin+6}:\radius-0.2cm) to ({360/\n * (\s - 1) + \adjust/\n+\margin-12}:\radius-0.3cm);
}
  \draw[decoration={markings,mark=at position 1 with {\arrow[scale=2]{>}}},
    postaction={decorate}, shorten >=0.4pt] ({360/\n * (1 +2)+\adjust/\n+\margin+6}:\radius-0.2cm) to ({360/\n * (1 - 1) + \adjust/\n+\margin-12}:\radius-0.3cm);
      \draw[decoration={markings,mark=at position 1 with {\arrow[scale=2]{>}}},
    postaction={decorate}, shorten >=0.4pt] ({360/\n * (5 +2)+\adjust/\n+\margin+6}:\radius-0.2cm) to ({360/\n * (5 - 1) + \adjust/\n+\margin-12}:\radius-0.3cm);

\node[] at ({360/\n * (0 )+ \correction/\n}:\radius + 0.2 cm) {$4$};

\node[] at ({360/\n * (2 )+ \correction/\n}:\radius + 0.2 cm) {$1$};
\node[] at ({360/\n * (3 )+ \correction/\n}:\radius + 0.2 cm) {$2$};
\node[] at ({360/\n * (4 )+ \correction/\n}:\radius + 0.2 cm) {$3$};

\node[ann] at (0.35, -0.35) {$x_2$}; 
\node[ann] at (0.45, -0.7) {$x_3$};
\node[ann] at (0.43, 0.05) {$x_2$}; 
\node[ann] at (0.8, 0.35) {$x_3$};
\node[ann] at (0, 0.5) {$x_2$}; 
\node[ann] at (0, 0.85) {$x_3$};
\node[ann] at (-0.43, 0.05) {$x_2$}; 
\node[ann] at (-0.8, 0.35) {$x_3$};
\node[ann] at (-0.35, -0.35) {$x_2$}; 
\node[ann] at (-0.45, -0.7) {$x_3$};
\node[ann] at (1.7, -0.4) {$x_1$};
\node[ann] at (-1.7, -0.4) {$x_1$};
\node[ann] at (0, -1.77) {$x_1$};
\node[ann] at (-1, 1.5) {$x_1$};
\node[ann] at (1, 1.5) {$x_1$};

 \tikzstyle{every node}=[draw, rectangle, fill = white, minimum size=5pt,
                            inner sep=2pt] 

\node[] at ({360/\n * (1 )+ \correction/\n}:\radius + 0.2 cm) {$0$};

\end{tikzpicture}    
\end{center}
and the relations are given by $x_ix_j-x_jx_i=0$. We endow the algebra $A$ with a grading by putting the thick arrows in degree $1$ and the other arrows in degree $0$. Let $e$ be the idempotent corresponding to vertex $0$. Then $A$ is $3$-AS-regular of Gorenstein parameter $2$. Moreover, $eAe\cong S^G$ is an isolated singularity, so $A/AeA$ is finite-dimensional. We also have that $S^G$ is $AS$-Gorenstein, because $G<SL(3,k$). Finally, $eA_0e' = e'A_0e = 0$, so, by Theorem \ref{thm:silting}, $\q e'Ae\oplus \q e'Ae(1)$ is a tilting object in $\sbd(\qgr eAe)$. We thus have a triangle equivalence 
\[
	\sing(S^G)\cong \bd((1-\tilde e)(\nabla A)(1-\tilde e)).
\]

The algebra $\nabla A$ has the following quiver 
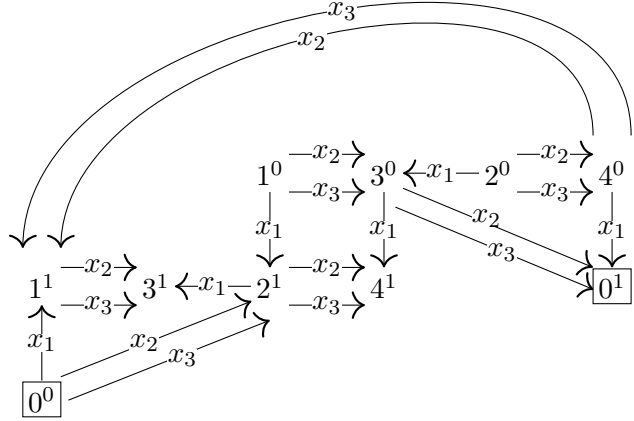
\begin{center}
\begin{tikzpicture}
 \tikzstyle{every node}=[draw,circle,fill=black,minimum size=5pt,
                            inner sep=0pt]
    \draw [decoration={markings,mark=at position 1 with {\arrow[scale=2]{>}}},
    postaction={decorate}, shorten >=0.4pt] (0,0.25) -- (1,0.25);  
      \draw [decoration={markings,mark=at position 1 with {\arrow[scale=2]{>}}},
    postaction={decorate}, shorten >=0.4pt] (0,-0.25) -- (1,-0.25);
    \draw [decoration={markings,mark=at position 1 with {\arrow[scale=2]{>}}},
    postaction={decorate}, shorten >=0.4pt] (2.5,0) -- (1.5,0);  
        \draw [decoration={markings,mark=at position 1 with {\arrow[scale=2]{>}}},
    postaction={decorate}, shorten >=0.4pt] (3,0.25) -- (4,0.25);  
        \draw [decoration={markings,mark=at position 1 with {\arrow[scale=2]{>}}},
    postaction={decorate}, shorten >=0.4pt] (3,-0.25) -- (4,-0.25);  
           \draw [decoration={markings,mark=at position 1 with {\arrow[scale=2]{>}}},
    postaction={decorate}, shorten >=0.4pt] (-0.25,-0.25) -- (-0.25,-1.25);  
           \draw [decoration={markings,mark=at position 1 with {\arrow[scale=2]{>}}},
    postaction={decorate}, shorten >=0.4pt] (1.25,-0.25) -- (1.25,-1.25);  
               \draw [decoration={markings,mark=at position 1 with {\arrow[scale=2]{>}}},
    postaction={decorate}, shorten >=0.4pt] (0,-1.25) -- (1,-1.25);  
           \draw [decoration={markings,mark=at position 1 with {\arrow[scale=2]{>}}},
    postaction={decorate}, shorten >=0.4pt] (0,-1.75) -- (1,-1.75);  
               \draw [decoration={markings,mark=at position 1 with {\arrow[scale=2]{>}}},
    postaction={decorate}, shorten >=0.4pt] (-0.5,-1.5) -- (-1.5,-1.5);  
               \draw [decoration={markings,mark=at position 1 with {\arrow[scale=2]{>}}},
    postaction={decorate}, shorten >=0.4pt] (-3,-1.75) -- (-2,-1.75);  
               \draw [decoration={markings,mark=at position 1 with {\arrow[scale=2]{>}}},
    postaction={decorate}, shorten >=0.4pt] (-3,-1.25) -- (-2,-1.25);  
                   \draw [decoration={markings,mark=at position 1 with {\arrow[scale=2]{>}}},
    postaction={decorate}, shorten >=0.4pt] (-3.25,-2.75) -- (-3.25,-1.75);  
                       \draw [decoration={markings,mark=at position 1 with {\arrow[scale=2]{>}}},
    postaction={decorate}, shorten >=0.4pt] (-3,-2.7) -- (-0.5,-1.7);  
                       \draw [decoration={markings,mark=at position 1 with {\arrow[scale=2]{>}}},
    postaction={decorate}, shorten >=0.4pt] (-2.9,-3.01) -- (-0.25,-1.95);  
                           \draw [decoration={markings,mark=at position 1 with {\arrow[scale=2]{>}}},
    postaction={decorate}, shorten >=0.4pt] (4.25,-0.25) -- (4.25,-1.25);
                               \draw [decoration={markings,mark=at position 1 with {\arrow[scale=2]{>}}},
    postaction={decorate}, shorten >=0.4pt] (1.5,-0.2) -- (4,-1.25);
                               \draw [decoration={markings,mark=at position 1 with {\arrow[scale=2]{>}}},
    postaction={decorate}, shorten >=0.4pt] (1.4,-0.45) -- (4,-1.54);

    	\draw[decoration={markings,mark=at position 0.999 with {\arrow[scale=2]{>}}}, postaction = {decorate}, shorten >=0.4 pt] (4,0.5) to [out = 90, in = 90] (-3, -1);
      \draw[decoration={markings,mark=at position 0.999 with {\arrow[scale=2]{>}}},
    postaction={decorate}, shorten >=0.4pt] (4.5,0.5)  to [out=90,in= 90] (-3.5,-1);

 \tikzstyle{every node}=[draw, circle, fill = white, draw = none, minimum size=1pt,
                            inner sep=0pt]      
     \draw (-0.25, 0) node{$1^0$};
     \draw (1.25, 0) node{$3^0$};
     \draw(2.75, 0) node{$2^0$};
     \draw(4.25, 0) node{$4^0$};
     \draw(-0.25, -1.5) node{$2^1$};
     \draw(1.25, -1.5) node{$4^1$};
      \draw(-1.75, -1.5) node{$3^1$};
       \draw(-3.25, -1.5) node{$1^1$};

\tikzstyle{every node}=[draw, rectangle, fill = white, draw = none, minimum size=5pt,
                            inner sep=0.1pt] 
                            
   \draw(0.3, 1.75) node {$x_2$};
   \draw(0.7, 2.15) node {$x_3$};
 \draw(3.5, 0.25) node {$x_2$}; 
 \draw(3.5, -0.25) node {$x_3$};         
 \draw(2, 0) node {$x_1$};
  \draw(0.5, 0.25) node {$x_2$}; 
 \draw(0.5, -0.25) node {$x_3$};
   \draw(0.5, -1.25) node {$x_2$}; 
 \draw(0.5, -1.75) node {$x_3$};
  \draw(1.25, -0.75) node {$x_1$};
    \draw(-0.25, -0.75) node {$x_1$};
     \draw(-1, -1.5) node {$x_1$};
        \draw(-2.5, -1.25) node {$x_2$}; 
 \draw(-2.5, -1.75) node {$x_3$};
  \draw(4.25, -0.75) node {$x_1$};
   \draw(-3.25, -2.25) node {$x_1$};
    \draw(-1.9, -2.25) node {$x_2$};
     \draw(-1.4, -2.45) node {$x_3$};
 \draw(2.6, -0.6) node {$x_2$};
     \draw(2.8, -1.05) node {$x_3$};

 \tikzstyle{every node}=[draw, rectangle, fill = white, minimum size=5pt,
                            inner sep=2pt] 
                            
               \draw(-3.25, -3) node {$0^0$};
       \draw (4.25, -1.5) node {$0^1$};

\end{tikzpicture}
\end{center}
and the relations are induced from the relations in $A$. The idempotent $\tilde e$ induced from $e$ corresponds to the vertices in the boxes. Thus, $(1-\tilde e)(\nabla A)(1-\tilde e)$ is given by the following quiver

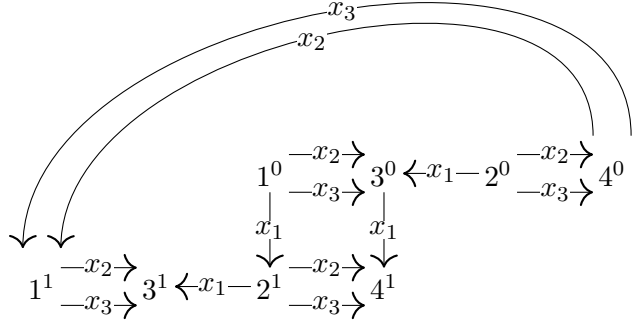
\begin{center}
\begin{tikzpicture}
 \tikzstyle{every node}=[draw,circle,fill=black,minimum size=5pt,
                            inner sep=0pt]
    \draw [decoration={markings,mark=at position 1 with {\arrow[scale=2]{>}}},
    postaction={decorate}, shorten >=0.4pt] (0,0.25) -- (1,0.25);  
      \draw [decoration={markings,mark=at position 1 with {\arrow[scale=2]{>}}},
    postaction={decorate}, shorten >=0.4pt] (0,-0.25) -- (1,-0.25);
    \draw [decoration={markings,mark=at position 1 with {\arrow[scale=2]{>}}},
    postaction={decorate}, shorten >=0.4pt] (2.5,0) -- (1.5,0);  
        \draw [decoration={markings,mark=at position 1 with {\arrow[scale=2]{>}}},
    postaction={decorate}, shorten >=0.4pt] (3,0.25) -- (4,0.25);  
        \draw [decoration={markings,mark=at position 1 with {\arrow[scale=2]{>}}},
    postaction={decorate}, shorten >=0.4pt] (3,-0.25) -- (4,-0.25);  
           \draw [decoration={markings,mark=at position 1 with {\arrow[scale=2]{>}}},
    postaction={decorate}, shorten >=0.4pt] (-0.25,-0.25) -- (-0.25,-1.25);  
           \draw [decoration={markings,mark=at position 1 with {\arrow[scale=2]{>}}},
    postaction={decorate}, shorten >=0.4pt] (1.25,-0.25) -- (1.25,-1.25);  
               \draw [decoration={markings,mark=at position 1 with {\arrow[scale=2]{>}}},
    postaction={decorate}, shorten >=0.4pt] (0,-1.25) -- (1,-1.25);  
           \draw [decoration={markings,mark=at position 1 with {\arrow[scale=2]{>}}},
    postaction={decorate}, shorten >=0.4pt] (0,-1.75) -- (1,-1.75);  
               \draw [decoration={markings,mark=at position 1 with {\arrow[scale=2]{>}}},
    postaction={decorate}, shorten >=0.4pt] (-0.5,-1.5) -- (-1.5,-1.5);  
               \draw [decoration={markings,mark=at position 1 with {\arrow[scale=2]{>}}},
    postaction={decorate}, shorten >=0.4pt] (-3,-1.75) -- (-2,-1.75);  
               \draw [decoration={markings,mark=at position 1 with {\arrow[scale=2]{>}}},
    postaction={decorate}, shorten >=0.4pt] (-3,-1.25) -- (-2,-1.25);  
    
    	\draw[decoration={markings,mark=at position 0.999 with {\arrow[scale=2]{>}}}, postaction = {decorate}, shorten >=0.4 pt] (4,0.5) to [out = 90, in = 90] (-3, -1);
      \draw[decoration={markings,mark=at position 0.999 with {\arrow[scale=2]{>}}},
    postaction={decorate}, shorten >=0.4pt] (4.5,0.5)  to [out=90,in= 90] (-3.5,-1);

 \tikzstyle{every node}=[draw, circle, fill = white, draw = none, minimum size=1pt,
                            inner sep=0pt]      
     \draw (-0.25, 0) node{$1^0$};
     \draw (1.25, 0) node{$3^0$};
     \draw(2.75, 0) node{$2^0$};
     \draw(4.25, 0) node{$4^0$};
     \draw(-0.25, -1.5) node{$2^1$};
     \draw(1.25, -1.5) node{$4^1$};
      \draw(-1.75, -1.5) node{$3^1$};
       \draw(-3.25, -1.5) node{$1^1$};
       
\tikzstyle{every node}=[draw, rectangle, fill = white, draw = none, minimum size=5pt,
                            inner sep=0.1pt] 
                            
   \draw(0.3, 1.75) node {$x_2$};
   \draw(0.7, 2.15) node {$x_3$};
 \draw(3.5, 0.25) node {$x_2$}; 
 \draw(3.5, -0.25) node {$x_3$};         
 \draw(2, 0) node {$x_1$};
  \draw(0.5, 0.25) node {$x_2$}; 
 \draw(0.5, -0.25) node {$x_3$};
   \draw(0.5, -1.25) node {$x_2$}; 
 \draw(0.5, -1.75) node {$x_3$};
  \draw(1.25, -0.75) node {$x_1$};
    \draw(-0.25, -0.75) node {$x_1$};
     \draw(-1, -1.5) node {$x_1$};
        \draw(-2.5, -1.25) node {$x_2$}; 
 \draw(-2.5, -1.75) node {$x_3$};

\end{tikzpicture}
\end{center}  
\end{example}


\section{A tilting object from levelled mutations}

In the previous section, we found a tilting object in some specific cases, and only when the Gorenstein parameter was $1$ or $2$. In this section, we use mutations to find a different tilting object in the case where $\nabla A$ is a Koszul levelled algebra. Our strategy is still to compare the semiorthogonal decompositions of Orlov and Minamoto--Mori. 

\begin{theorem}
\label{thm:main}
Let $A= \oplus_{i\geq 0} A_i$ and $e$ be as in Setting \ref{setting}. In addition, suppose that $eA_0e\cong k$ and that $\nabla A$ is a Koszul levelled algebra. There is a triangle equivalence 
\[
	\sing(eAe)\cong \bd((1-\tilde e)(\nabla A)^!(1-\tilde e)).
\]
\end{theorem}

\begin{remark}
As opposed to the situation in Theorem \ref{thm:MU}, the algebra $A$ itself is not necessarily Koszul with respect to the grading that endows it with a structure of AS-regular algebra. 
\end{remark}

\begin{proof}
The idea is to compare two semiorthogonal decompositions of $\bd(\qgr eAe)$, mentioned in the preliminaries: 
\begin{itemize}
\item \cite[Theorem 4.12]{MM11} ${\displaystyle \bd(\qgr eAe) = \langle \q Ae, \q Ae(1), \ldots, \q Ae(\ell-1)\rangle}$;
\item \cite[Theorem 2.5]{Orl09} ${\displaystyle \bd(\qgr eAe) = \langle \q eAe, \q eAe(1), \ldots, \q eAe(\ell-1), \Phi(\sing(eAe))\rangle}$.
\end{itemize}
Recall from Theorem \ref{thm:MM_tilting} that
\[
	\Hom_{\bd(\qgr eAe)} (\oplus_{i=0}^{\ell-1}\q Ae(i), \oplus_{i=0}^{\ell-1}\q Ae(i)) \cong \nabla A.
\]	
Moreover, the object $T:=\oplus_{i=0}^{\ell-1}\q Ae(i)$ is tilting, so there is a triangle equivalence 
\[
	F := \RHom_{\bd(\qgr eAe)}(T, -): \bd(\qgr eAe)\xrightarrow{\sim} \bd(\nabla A).
\] 
Since $\nabla A$ is levelled, the indecomposable projective $\nabla A$-modules give rise to a levelled full strong exceptional collection 
\[
	\bd(\nabla A) = \langle P_0, \ldots, P_m \rangle =: \mathbb P.
\]
This induces a levelled structure on a full strong exceptional collection 
\[
	\bd(\qgr eAe) = \langle E_0, \ldots, E_m \rangle = \langle \mathbb E_0,\ldots, \mathbb E_n\rangle =: \mathbb E,
\]
where each $E_j = F^{-1}(P_j)$ is an exceptional direct summand of $\q Ae(r)$ for some $0\leq r\leq \ell-1$, and $\E_j$ is the collection at level $j$. In particular, since 
\[
	\Hom_{\bd(\qgr eAe)}(\q eAe, \q eAe)\cong eA_0e \cong k, 
\]
we have that $\q eAe(i)$ is exceptional for all $0\leq i\leq \ell-1$, so $\q eAe(i) = E_{j_i}$ for some $0\leq j_i\leq m$. \\

Denote by $s: \{0,\ldots, m\}\to \{0,\ldots, n\}$ the level function on this collection. We now perform right mutations on $\mathbb \E$ to obtain a full exceptional collection
\[
	R_0\E:=  (\q eAe, \mathbb E_{s(j_0) +1}, \ldots, \E_n, R^{n-s(j_0)}(\E_{s(j_0)}'), R^{n-s(j_0)+1}(\E_{s(j_0)-1}), \ldots,  R^{n}(\E_0))
\]	
which generates $\bd(\qgr eAe)$. Here, $ \mathbb E_{s(j_0)}'$ is the collection obtained from $\mathbb E_{s(j_0)}$ by removing $\q eAe$. Similarly, denote by $R'_0\E$ the subcollection consisting of the objects of $R_0\E$, except for $\q eAe$. Then, by Lemma \ref{lem:orl_quot}, 
\[
	\langle R'_0\E\rangle \cong {^{\perp}\langle} \q eAe\rangle \cong \bd(\qgr eAe)/\langle \q eAe\rangle.
\]	
Comparing with Orlov's semiorthogonal decomposition, we thus conclude that 
\[
	\langle R'_0\E\rangle\cong \langle \q eAe(1), \ldots, \q eAe(\ell-1), \Phi(\sing(eAe))\rangle.
\]	
We now do right mutations on $R_0'\E$ to obtain a full exceptional collection 
\begin{multline*}
	R_1R_0'\E:= (\q eAe(1), \mathbb E_{s(j_1) +1}, \ldots, \E_n, R^{n-s(j_1)}(\E_{s(j_1)}'), R^{n-s(j_1)+1}(\E_{s(j_1)-1}), \ldots,\\  R^{n-s(j_0)-1}(\E_{s(j_0)+1}), R^{n-s(j_0)}(\E_{s(j_0)}'), R^{n-s(j_0)+1}(\E_{s(j_0)-1}), \ldots,  R^n(\E_0)).
\end{multline*} 
Then, by the same reasoning, we have an equivalence 
\[
	\langle R_1'R_0'\E\rangle\cong \langle \q eAe(2), \ldots, \q eAe(\ell-1), \Phi(\sing(eAe))\rangle.
\]
Continuing in this fashion for every $\q eAe(i)$, we obtain an equivalence 
\[
	\langle R'_{\ell-1}\cdots R'_0\E\rangle \cong \bd(\qgr eAe)/\langle \q eAe, \ldots, \q eAe(\ell-1)\rangle\cong \sing(eAe).
\]	

Note that $R'_{\ell-1}\cdots R'_0\E$ is a subcollection of the right dual collection $\E^{\vee}$, from which the objects $R^{n-s(j_i)}(\q eAe(i))$ are removed. By Corollary \ref{cor:kos_dual}, since $\nabla A$ is levelled Koszul, we have that $\E^{\vee}$ is a full strong exceptional collection and 
\[
	\End(\E^{\vee}) \cong (\nabla A)^!.
\]
Therefore the collection $R'_{\ell-1}\cdots R'_0\E$ is also strong. Now, by uniqueness of the Serre functor, there is a commutative diagram 

\begin{center}
\begin{tikzcd}
  \bd(\qgr eAe) \arrow[r, "F"] \arrow[d, "\Se_m^{-1}" ]
    & \bd(\nabla A) \arrow[d, "\Se_m^{-1}"] \\
  \bd(\qgr eAe) \arrow[r, "F" ]
&\bd(\nabla A) \end{tikzcd}
\end{center}
where, by abuse of notation, $\Se_m = \Se[-m]$ denotes the shifted Serre functor in $\bd(\qgr eAe)$ and $\bd(\nabla A)$. Moreover, 
\begin{align*}
	F(\q eAe(i)) &= \Hom_{\bd(\qgr eAe)}(\oplus_{i=0}^{\ell-1}\q Ae(i), \q eAe(i))\\
	&\cong \tilde e_i \Hom_{\bd(\qgr eAe)}(\oplus_{i=0}^{\ell-1}\q Ae(i), \oplus_{i=0}^{\ell-1}\q Ae(i))\\
	&\cong \tilde e_i \nabla A,
\end{align*}
where $\tilde e_i$ is defined in \ref{def:beilin}. Thus,
\[
	R^{n-s(j_i)}(\q eAe(i))\cong \Se_m^{-1}(L^{s(j_i)}(\q eAe(i)))\xmapsto{F} \Se_m^{-1}(L^{s(j_i)}(\tilde e_i\nabla A)) \cong \Se_m^{-1}(\topm(\tilde e_i\nabla A)[-s(j_i)]).
\] 
We conclude that 
\begin{align*}
	\End(R'_{\ell-1}\cdots R'_0\E)&\cong\End_{\bd(\nabla A)}\left(\bigoplus_{\substack{j\,:\, j\not = j_i}} \topm(P_j)[-s(j)]\right)\\
	&\cong (1-\tilde e)(\nabla A)^!(1-\tilde e), 
\end{align*}
so, applying theorem \ref{thm:keller}, we obtain a triangle equivalence 
\[
	\sing(eAe)\cong\bd((1-\tilde e)(\nabla A)^!(1-\tilde e)).
\]	
\end{proof}

Perhaps the most interesting consequence of this theorem is that the equivalence of Mori--Ueyama, cited in Theorem \ref{thm:MU}, can be obtained as a special case.  

\begin{corollary}[{\cite[Theorem 4.17]{MU16}}]
Let $S$ be a noetherian AS-regular Koszul algebra over $k$ of dimension $d\geq 2$. Let $G \leq \GrAut S$ be a finite subgroup such that $\chare  k$ does not divide $|G|$ and let $e = \frac{1}{|G|} \sum_{g\in G} g$. Finally, assume that $S^G \cong eS\#Ge$ is AS-Gorenstein and $S\#G/( e)$ is finite-dimensional over $k$. Then there is a triangle equivalence 
\[
	\sing(S^G)\cong \bd((1-\tilde e)(G\# \nabla (S^!))(1-\tilde e)).
\]
 \end{corollary}
 
 \begin{proof}
 By \cite[Lemma 2.21]{MU16}, the algebra $A:=S\#G$ is AS-regular, so $A$ and the idempotent $e$ satisfy the conditions of Setting \ref{setting}. Moreover, we have that $eA_0e = ekGe\cong k$. Also, $S$ is Koszul, so $\nabla A \cong (\nabla S)\#G$ is Koszul.\\
 
 Since $A$ itself is also Koszul, there is a graded isomorphism $A\cong T_{A_0} A_1/(R)$. In particular, the Gorenstein parameter of $A$ is $d$. Moreover, $A$ is Morita equivalent to a quiver algebra $\bar A = T_{kQ_0} kQ_1/(\bar R)$ and $\nabla A$ is Morita equivalent to $\nabla \bar A = T_{k\tilde Q_0}k\tilde Q_1/(\tilde R)$. Suppose that $Q_0 = \{f^0,\ldots, f^m\}$ and $\tilde Q_0:=\{\tilde f^0_i, \ldots, \tilde f^m_i\,\,|\,\, 0\leq i\leq d-1\}$, where $\tilde f_i^t$ is an idempotent induced from $f^t$ as defined in \ref{def:beilin}. We define a level function $s: \tilde Q_0\to \{0,\ldots, d-1\}$ on $\nabla \bar A$ by setting $s(\tilde f^t_i) := i$ for all $0\leq t\leq m$ and $0\leq i\leq d-1$. This is in fact a level function, because 
 \[
 	\tilde f^r_j (\nabla \bar A) \tilde f^t_i \cong (f^rkQ_{j-i}f^t)_{(j,i)},
 \]
 where $(f^rkQ_{j-i}f^t)_{(j,i)}$ is the matrix whose only non-zero entry is at position $(j, i)$ and consists of the space generated by paths of length $j-i$ in $\bar A$ starting at $f^t$ and ending at $f^r$. This implies that $\tilde f^r_j (\nabla \bar A) \tilde f^t_i$ contains arrows only when $j = i+1$. Thus, if $\tilde f^r_j \tilde Q_1 \tilde f^t_i\not = \varnothing$, then $s(j) = s(i) + 1$. Therefore, all the hypotheses of our main theorem are fulfilled. \\
 
 Finally, by \cite[Proposition 2.14]{MU16}, 
 \[
 	(\nabla(S\#G))^!\cong ((\nabla S)\#G)^!\cong G\#(\nabla S)^! \cong G\#\nabla (S^!),
 \]
 so the conclusions of Theorems \ref{thm:MU} and \ref{thm:main} are the same in this context. 
 \end{proof}
 
 We conclude with an example where the Gorenstein parameter is not equal to $1$ nor to the global dimension. 
 
 \begin{example}
 Let $A = k[x_1, x_2, x_3, x_4]\#G$, where $G = \frac{1}{4}(1,1,3,3)<SL(4,k)$. The following quivers are described in Definition \ref{def:mckay}. The quiver of $A$ is given by: 

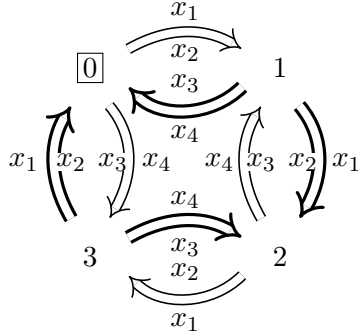
\begin{center}
\begin{tikzpicture}[scale = 2.5]
 \tikzstyle{every node}=[draw, circle, fill = white, draw = none, minimum size=1pt,
                            inner sep=0pt] 

\draw (1,0) node {$1$};
\draw (1,-1) node {$2$};
\draw (0,-1) node {$3$};

 \tikzstyle{every node}=[draw, rectangle, fill = white, minimum size=5pt,
                            inner sep=2pt] 
                            
\draw (0,0) node {$0$};

 \tikzstyle{every node}=[draw,circle,fill=black,minimum size=5pt,
                            inner sep=0pt]
  \draw [-Implies, line width = 0.6pt, double distance=3pt, shorten >=0.4pt] (0.2,0.1) to [out = 30, in = 150] (0.8,0.1);  
      \draw [-Implies, line width = 0.6pt, double distance=3pt, shorten >=0.4pt, very thick] (0.8,-0.1) to [out = 225, in = 315] (0.2,-0.1); 
      \draw [-Implies, line width = 0.6pt, double distance=3pt, shorten >=0.4pt, very thick] (-0.1,-0.8) to [out = 120, in = 240] (-0.1,-0.2); 
      \draw [-Implies, line width = 0.6pt, double distance=3pt, shorten >=0.4pt] (0.1,-0.2) to [out = 310, in = 60] (0.1,-0.8); 
        \draw [-Implies, line width = 0.6pt, double distance=3pt, shorten >=0.4pt, very thick] (0.2,-0.9) to [out = 30, in = 150] (0.8,-0.9);  
      \draw [-Implies, line width = 0.6pt, double distance=3pt, shorten >=0.4pt] (0.8,-1.1) to [out = 225, in = 315] (0.2,-1.1);     
            \draw [-Implies, line width = 0.6pt, double distance=3pt, shorten >=0.4pt] (0.9,-0.8) to [out = 120, in = 240] (0.9,-0.2); 
      \draw [-Implies, line width = 0.6pt, double distance=3pt, shorten >=0.4pt, very thick] (1.1,-0.2) to [out = 310, in = 60] (1.1,-0.8);    
      
\tikzstyle{every node}=[draw, rectangle, fill = white, draw = none, minimum size=5pt,
                            inner sep=0.1pt] 

\draw (0.5, 0.3) node {$x_1$};
\draw (0.5, 0.08) node {$x_2$};
\draw (0.5, -0.08) node {$x_3$};
\draw (0.5, -0.35) node {$x_4$};
\draw (0.5, -0.7) node {$x_4$};
\draw (0.5, -0.95) node {$x_3$};
\draw (0.5, -1.08) node {$x_2$};
\draw (0.5, -1.35) node {$x_1$};
\draw (-0.35, -0.5) node {$x_1$};
\draw (-0.1, -0.5) node {$x_2$};
\draw (0.12, -0.5) node {$x_3$};
\draw (0.35, -0.5) node {$x_4$};
\draw (0.9, -0.5) node {$x_3$};
\draw (0.68, -0.5) node {$x_4$};
\draw (1.12, -0.5) node {$x_2$};
\draw (1.35, -0.5) node {$x_1$};

\end{tikzpicture}  
\end{center} 
with relations $x_ix_j - x_jx_i =0$. We give $A$ a grading by putting the thick arrows in degree $1$ and the other arrows in degree $0$. In this case, $A$ is a $4$-AS-regular algebra of Gorenstein parameter $2$. It is important to mention that $A$ is not Koszul with respect to the grading that endows it with the structure of an AS-regular algebra, so it does not fit in the setting of \cite{MU16}.\\

 Let $e$ be the idempotent corresponding to vertex $0$. Then $eAe\cong S^G$ is an isolated singularity, so $A/AeA$ is finite-dimensional. Moreover, since $G<SL(4,k)$, $S^G$ is AS-Gorenstein. The quiver of $\nabla A$ is given by

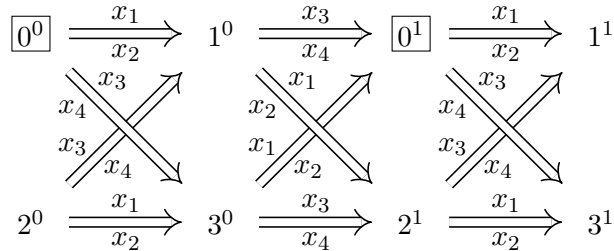
\begin{center}
\begin{tikzpicture}[scale = 2.5]
 \tikzstyle{every node}=[draw, circle, fill = white, draw = none, minimum size=1pt,
                            inner sep=0pt] 

\draw (1,0) node {$1^0$};
\draw (0,-1) node {$2^0$};
\draw (1,-1) node {$3^0$};
\draw (3,0) node {$1^1$};
\draw (2,-1) node {$2^1$};
\draw (3,-1) node {$3^1$};

 \tikzstyle{every node}=[draw, rectangle, fill = white, minimum size=5pt,
                            inner sep=2pt] 
                            
\draw (0,0) node {$0^0$};
\draw (2,0) node {$0^1$};

\tikzstyle{every node}=[draw,circle,fill=black,minimum size=5pt,
                            inner sep=0pt]
  \draw [-Implies, line width = 0.6pt, double distance=3pt, shorten >=0.4pt] (0.2,0) -- (0.8,0); 
  \draw [-Implies, line width = 0.6pt, double distance=3pt, shorten >=0.4pt] (1.2,0) -- (1.8,0); 
  \draw [-Implies, line width = 0.6pt, double distance=3pt, shorten >=0.4pt] (2.2,0) -- (2.8,0); 
  \draw [-Implies, line width = 0.6pt, double distance=3pt, shorten >=0.4pt] (0.2,-1) -- (0.8,-1); 
  \draw [-Implies, line width = 0.6pt, double distance=3pt, shorten >=0.4pt] (1.2,-1) -- (1.8,-1); 
  \draw [-Implies, line width = 0.6pt, double distance=3pt, shorten >=0.4pt] (2.2,-1) -- (2.8,-1); 
  \draw [-Implies, line width = 0.6pt, double distance=3pt, shorten >=0.4pt] (0.2,-0.8) -- (0.8,-0.2); 
  \draw [-Implies, line width = 0.6pt, double distance=3pt, shorten >=0.4pt] (0.2,-0.2) -- (0.8,-0.8); 
  \draw [-Implies, line width = 0.6pt, double distance=3pt, shorten >=0.4pt] (1.2,-0.8) -- (1.8,-0.2); 
  \draw [-Implies, line width = 0.6pt, double distance=3pt, shorten >=0.4pt] (1.2,-0.2) -- (1.8,-0.8); 
  \draw [-Implies, line width = 0.6pt, double distance=3pt, shorten >=0.4pt] (2.2,-0.8) -- (2.8,-0.2); 
  \draw [-Implies, line width = 0.6pt, double distance=3pt, shorten >=0.4pt] (2.2,-0.2) -- (2.8,-0.8); 
  
  \tikzstyle{every node}=[draw, rectangle, fill = white, draw = none, minimum size=5pt,
                            inner sep=0.1pt] 

\draw (0.5, 0.1) node {$x_1$};
\draw (0.5, -0.1) node {$x_2$};
\draw (1.5, 0.1) node {$x_3$};
\draw (1.5, -0.1) node {$x_4$};
\draw (2.5, 0.1) node {$x_1$};
\draw (2.5, -0.1) node {$x_2$};
\draw (0.5, -0.9) node {$x_1$};
\draw (0.5, -1.1) node {$x_2$};
\draw (1.5, -0.9) node {$x_3$};
\draw (1.5, -1.1) node {$x_4$};
\draw (2.5, -0.9) node {$x_1$};
\draw (2.5, -1.1) node {$x_2$};

\draw (0.43, -0.25) node {$x_3$};
\draw (0.22, -0.4) node {$x_4$};
\draw (0.22, -0.6) node {$x_3$};
\draw (0.46, -0.72) node {$x_4$};
\draw (1.43, -0.25) node {$x_1$};
\draw (1.22, -0.4) node {$x_2$};
\draw (1.22, -0.6) node {$x_1$};
\draw (1.46, -0.72) node {$x_2$};
\draw (2.43, -0.25) node {$x_3$};
\draw (2.22, -0.4) node {$x_4$};
\draw (2.22, -0.6) node {$x_3$};
\draw (2.46, -0.72) node {$x_4$};
\end{tikzpicture}
\end{center}
and the relations are induced from the relations in $A$. The induced idempotent $\tilde e$ is the one corresponding to the vertices in the boxes. This is a Koszul levelled algebra, so by Theorem \ref{thm:main}, there is a triangle equivalence 
\[
	\sing(S^G)\cong\bd((1-\tilde e)(\nabla A)^!(1-\tilde e)). 
\]

The quiver of $(1-\tilde e)(\nabla A)^{!}(1-\tilde e)$ is given by

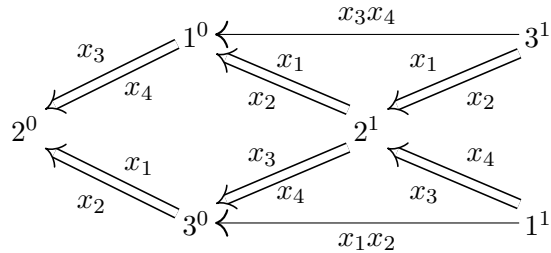
\begin{center}
\begin{tikzpicture}[scale = 2.5]
 \tikzstyle{every node}=[draw, circle, fill = white, draw = none, minimum size=1pt,
                            inner sep=0pt] 

\draw (0,-0.5) node {$2^0$};
\draw (0.9,0) node {$1^0$};
\draw (0.9,-1) node {$3^0$};
\draw (1.8,-0.5) node {$2^1$};
\draw (2.7,0) node {$3^1$};
\draw (2.7,-1) node {$1^1$};

\tikzstyle{every node}=[draw,circle,fill=black,minimum size=5pt,
                            inner sep=0pt]
  \draw [decoration={markings,mark=at position 1 with {\arrow[scale=2]{>}}},
    postaction={decorate}, shorten >=0.4pt] (2.6,0) -- (1,0); 
      \draw [decoration={markings,mark=at position 1 with {\arrow[scale=2]{>}}},
    postaction={decorate}, shorten >=0.4pt] (2.6,-1) -- (1,-1); 
    \draw [-Implies, line width = 0.6pt, double distance=3pt, shorten >=0.4pt] (2.6,-0.1) -- (1.9,-0.4); 
  \draw [-Implies, line width = 0.6pt, double distance=3pt, shorten >=0.4pt] (2.6,-0.9) -- (1.9,-0.6); 
  \draw [-Implies, line width = 0.6pt, double distance=3pt, shorten >=0.4pt] (1.7,-0.4) -- (1,-0.1); 
  \draw [-Implies, line width = 0.6pt, double distance=3pt, shorten >=0.4pt] (1.7,-0.6) -- (1,-0.9); 
  \draw [-Implies, line width = 0.6pt, double distance=3pt, shorten >=0.4pt] (0.8,-0.05) -- (0.1,-0.4); 
  \draw [-Implies, line width = 0.6pt, double distance=3pt, shorten >=0.4pt] (0.8,-0.95) -- (0.1,-0.6); 

\tikzstyle{every node}=[draw, rectangle, fill = white, draw = none, minimum size=5pt,
                            inner sep=0.1pt] 

\draw (1.8, 0.1) node {$x_3x_4$};
\draw (1.8, -1.1) node {$x_1x_2$};
\draw (0.35, -0.1) node {$x_3$};
\draw (0.6, -0.3) node {$x_4$};
\draw (0.6, -0.7) node {$x_1$};
\draw (0.35, -0.9) node {$x_2$};
\draw (1.4, -0.15) node {$x_1$};
\draw (1.25, -0.35) node {$x_2$};
\draw (1.25, -0.65) node {$x_3$};
\draw (1.4, -0.85) node {$x_4$};
\draw (2.1, -0.15) node {$x_1$};
\draw (2.4, -0.35) node {$x_2$};
\draw (2.1, -0.85) node {$x_3$};
\draw (2.4, -0.65) node {$x_4$};

\end{tikzpicture}
\end{center}
and the relations are given by $x_ix_j + x_jx_i =0$. 
 \end{example}


\bibliographystyle{alpha}

\bibliography{bib_stable_MCM_diff_gor_par.bib}

\def\cprime{$'$}
\begin{thebibliography}{HIMO14}

\bibitem[AIR15]{AIR15}
Claire Amiot, Osamu Iyama, and Idun Reiten.
\newblock Stable categories of {C}ohen-{M}acaulay modules and cluster
  categories.
\newblock {\em Amer. J. Math.}, 137(3):813--857, 2015.

\bibitem[Ami13]{Ami13}
Claire Amiot.
\newblock Preprojective algebras, singularity categories and orthogonal
  decompositions.
\newblock In {\em Algebras, quivers and representations}, volume~8 of {\em Abel
  Symp.}, pages 1--11. Springer, Heidelberg, 2013.

\bibitem[BGS96]{BGS96}
Alexander Beilinson, Victor Ginzburg, and Wolfgang Soergel.
\newblock Koszul duality patterns in representation theory.
\newblock {\em J. Amer. Math. Soc.}, 9(2):473--527, 1996.

\bibitem[BIY18]{BIY18}
Ragnar-Olaf Buchweitz, Osamu Iyama, and Kota Yamaura.
\newblock Tilting theory for gorenstein rings in dimension one, 2018.
\newblock arXiv:1803.05269.

\bibitem[BK89]{BK89}
A.~I. Bondal and M.~M. Kapranov.
\newblock Representable functors, {S}erre functors, and reconstructions.
\newblock {\em Izv. Akad. Nauk SSSR Ser. Mat.}, 53(6):1183--1205, 1337, 1989.

\bibitem[Bon89]{Bon89}
A.~I. Bondal.
\newblock Representations of associative algebras and coherent sheaves.
\newblock {\em Izv. Akad. Nauk SSSR Ser. Mat.}, 53(1):25--44, 1989.

\bibitem[BS10]{BS10}
Tom Bridgeland and David Stern.
\newblock Helices on del {P}ezzo surfaces and tilting {C}alabi-{Y}au algebras.
\newblock {\em Adv. Math.}, 224(4):1672--1716, 2010.

\bibitem[BS15]{BS15}
Jesse Burke and Greg Stevenson.
\newblock The derived category of a graded {G}orenstein ring.
\newblock In {\em Commutative algebra and noncommutative algebraic geometry.
  {V}ol. {II}}, volume~68 of {\em Math. Sci. Res. Inst. Publ.}, pages 93--123.
  Cambridge Univ. Press, New York, 2015.

\bibitem[Buc87]{Buc87}
Ragnar-Olaf Buchweitz.
\newblock Maximal {C}ohen--{M}acaulay modules and {T}ate cohomology over
  {G}orenstein rings, 1987.
\newblock unpublished manuscript.

\bibitem[DL16a]{DL16a}
Laurent Demonet and Xueyu Luo.
\newblock Ice quivers with potential arising from once-punctured polygons and
  {C}ohen-{M}acaulay modules.
\newblock {\em Publ. Res. Inst. Math. Sci.}, 52(2):141--205, 2016.

\bibitem[DL16b]{DL16b}
Laurent Demonet and Xueyu Luo.
\newblock Ice quivers with potential associated with triangulations and
  {C}ohen-{M}acaulay modules over orders.
\newblock {\em Trans. Amer. Math. Soc.}, 368(6):4257--4293, 2016.

\bibitem[Han19]{Han19}
Norihiro Hanihara.
\newblock Yoneda algebras and their singularity categories, 2019.
\newblock arXiv:1902.09441.

\bibitem[Hil95]{Hil95}
Lutz Hille.
\newblock Consistent algebras and special tilting sequences.
\newblock {\em Math. Z.}, 220(2):189--205, 1995.

\bibitem[HIMO14]{HIMO14}
Martin Herschend, Osamu Iyama, Hiroyuki Minamoto, and Steffen Oppermann.
\newblock Representation theory of geigle-lenzing complete intersections, 2014.
\newblock arXiv:1409.0668.

\bibitem[IO13]{IO13}
Osamu Iyama and Steffen Oppermann.
\newblock Stable categories of higher preprojective algebras.
\newblock {\em Adv. Math.}, 244:23--68, 2013.

\bibitem[IT13]{IT13}
Osamu Iyama and Ryo Takahashi.
\newblock Tilting and cluster tilting for quotient singularities.
\newblock {\em Math. Ann.}, 356(3):1065--1105, 2013.

\bibitem[IY18]{IY18}
Osamu Iyama and Dong Yang.
\newblock Silting reduction and {C}alabi-{Y}au reduction of triangulated
  categories.
\newblock {\em Trans. Amer. Math. Soc.}, 370(11):7861--7898, 2018.

\bibitem[JZ00]{JZ00}
Peter J{\o}rgensen and James~J. Zhang.
\newblock Gourmet's guide to {G}orensteinness.
\newblock {\em Adv. Math.}, 151(2):313--345, 2000.

\bibitem[Kel94]{Kel94}
Bernhard Keller.
\newblock Deriving {DG} categories.
\newblock {\em Ann. Sci. \'Ecole Norm. Sup. (4)}, 27(1):63--102, 1994.

\bibitem[Kim18]{Kim18}
Yuta Kimura.
\newblock Tilting theory of preprojective algebras and {$c$}-sortable elements.
\newblock {\em J. Algebra}, 503:186--221, 2018.

\bibitem[Kon95]{Kon95}
Maxim Kontsevich.
\newblock Homological algebra of mirror symmetry.
\newblock In {\em Proceedings of the {I}nternational {C}ongress of
  {M}athematicians, {V}ol. 1, 2 ({Z}\"{u}rich, 1994)}, pages 120--139.
  Birkh\"{a}user, Basel, 1995.

\bibitem[KST07]{KST07}
Hiroshige Kajiura, Kyoji Saito, and Atsushi Takahashi.
\newblock Matrix factorization and representations of quivers. {II}. {T}ype
  {$ADE$} case.
\newblock {\em Adv. Math.}, 211(1):327--362, 2007.

\bibitem[KST09]{KST09}
Hiroshige Kajiura, Kyoji Saito, and Atsushi Takahashi.
\newblock Triangulated categories of matrix factorizations for regular systems
  of weights with {$\epsilon=-1$}.
\newblock {\em Adv. Math.}, 220(5):1602--1654, 2009.

\bibitem[LW12]{LW12}
Graham~J. Leuschke and Roger Wiegand.
\newblock {\em Cohen-{M}acaulay representations}, volume 181 of {\em
  Mathematical Surveys and Monographs}.
\newblock American Mathematical Society, Providence, RI, 2012.

\bibitem[Min12]{Min12}
Hiroyuki Minamoto.
\newblock Ampleness of two-sided tilting complexes.
\newblock {\em Int. Math. Res. Not. IMRN}, (1):67--101, 2012.

\bibitem[MM11]{MM11}
Hiroyuki Minamoto and Izuru Mori.
\newblock The structure of {AS}-{G}orenstein algebras.
\newblock {\em Adv. Math.}, 226(5):4061--4095, 2011.

\bibitem[MU16a]{MU16b}
Izuru Mori and Kenta Ueyama.
\newblock Ample group action on {AS}-regular algebras and noncommutative graded
  isolated singularities.
\newblock {\em Trans. Amer. Math. Soc.}, 368(10):7359--7383, 2016.

\bibitem[MU16b]{MU16}
Izuru Mori and Kenta Ueyama.
\newblock Stable categories of graded maximal {C}ohen-{M}acaulay modules over
  noncommutative quotient singularities.
\newblock {\em Adv. Math.}, 297:54--92, 2016.

\bibitem[Orl04]{Orl04}
Dmitri Orlov.
\newblock Triangulated categories of singularities and {D}-branes in
  {L}andau-{G}inzburg models.
\newblock {\em Tr. Mat. Inst. Steklova}, 246(Algebr. Geom. Metody, Svyazi i
  Prilozh.):240--262, 2004.

\bibitem[Orl09]{Orl09}
Dmitri Orlov.
\newblock Derived categories of coherent sheaves and triangulated categories of
  singularities.
\newblock In {\em Algebra, arithmetic, and geometry: in honor of {Y}u. {I}.
  {M}anin. {V}ol. {II}}, volume 270 of {\em Progr. Math.}, pages 503--531.
  Birkh\"auser Boston, Inc., Boston, MA, 2009.

\bibitem[Thi20]{Thi20}
Louis-Philippe Thibault.
\newblock Preprojective algebra structure on skew-group algebras.
\newblock {\em Adv. Math.}, 365:107033, 2020.

\bibitem[Ued08]{Ued08}
Kazushi Ueda.
\newblock Triangulated categories of {G}orenstein cyclic quotient
  singularities.
\newblock {\em Proc. Amer. Math. Soc.}, 136(8):2745--2747, 2008.

\bibitem[Yam13]{Yam13}
Kota Yamaura.
\newblock Realizing stable categories as derived categories.
\newblock {\em Adv. Math.}, 248:784--819, 2013.

\bibitem[Yos90]{Yos90}
Yuji Yoshino.
\newblock {\em Cohen-{M}acaulay modules over {C}ohen-{M}acaulay rings}, volume
  146 of {\em London Mathematical Society Lecture Note Series}.
\newblock Cambridge University Press, Cambridge, 1990.

\end{thebibliography}

\end{document}